\documentclass[11pt, reqno]{article}

\usepackage[a4paper, left=2.5cm, right=2.5cm, top=2 cm, bottom= 2 cm]{geometry}

\usepackage{amsmath, amsthm, amssymb}
\usepackage{epsfig}
\usepackage{verbatim}
\usepackage{multicol}
\usepackage{url}
\usepackage{latexsym}
\usepackage{mathrsfs}
\usepackage[colorlinks, bookmarks=true]{hyperref}
\usepackage{graphicx}
\usepackage{amsmath}
\usepackage{geometry}
\usepackage{enumitem}

\usepackage{color}

\usepackage{chngcntr}
\counterwithin*{equation}{section}  

\usepackage{cite}

\begin{document}
\newcommand{\mb}{\mathbb}
\newcommand{\mf}{\mathfrak}
\newcommand{\mc}{\mathcal}
\newcommand{\mbf}{\mathbf}
\newcommand{\im}{{\rm im}}
\newcommand{\del}{\partial}
\newcommand{\coker}{{\rm coker}}
\newcommand{\Tr}{\operatorname{Tr}}
\newcommand{\vol}{\operatorname{vol}}
\newcommand{\area}{\operatorname{area}}
\newcommand{\diag}{\text{diag}}

\newcommand{\ZZ}{\mathbb{Z}}
\newcommand{\NN}{\mathbb{N}}
\newcommand{\QQ}{\mathbb{Q}}
\newcommand{\RR}{\mathbb{R}}
\newcommand{\CC}{\mathbb{C}}
\newcommand{\EE}{\mathbb{E}}
\newcommand{\PP}{\mathbb{P}}
\newcommand{\TT}{\mathbb{T}}

\def\eps{\epsilon}
\def\nn{\mathbb{N}}
\def\e{\mathrm{e}}
\def\mAt{m_{A, t}}
\def\i{\mathrm{i}}
\def\rhoAt{\rho_{A, t}}
\def\rhosc{\rho_{\mathrm{sc}}}
\def\gamAit{\gamma_{A, i, t}}
\def\D{\mathcal{D}}
\def\msc{m_{\mathrm{sc}}}
\def\d{\mathrm{d}}
\def\O{\mathcal{O}}
\def\F{\mathcal{F}}
\def\rhoHonet{\rho^{(H)}_{ t} }
\def\gamHonet{\gamma^{(H)}_{i, t} }
\def\ev{\boldsymbol{e}}
\def\thetat{\vartheta_t}
\def\spc{\mathrm{spec}}
\def\mtwot{m_{2, t}}
\def\mutwot{\mu_{2, t}}
\def\f{f}
\def\M{B}
\def\m{b}

\theoremstyle{plain}
  \newtheorem{theorem}{Theorem}[section]
  \newtheorem{proposition}[theorem]{Proposition}
  \newtheorem*{propositionnonumber}{Proposition}
  \newtheorem{lemma}[theorem]{Lemma}
  \newtheorem{corollary}[theorem]{Corollary}
  \newtheorem{conjecture}[theorem]{Conjecture}
\theoremstyle{definition}
  \newtheorem{definition}[theorem]{Definition}
  \newtheorem{example}[theorem]{Example}
  \newtheorem{examples}[theorem]{Examples}
  \newtheorem{question}[theorem]{Question}
  \newtheorem{problem}[theorem]{Problem}
  \newtheorem{assumption}[theorem]{Assumption}
  \newtheorem{remark}[theorem]{Remark}

\newcommand{\cor}{\color{red}}

\begin{table}
\centering

\begin{tabular}{c}
\multicolumn{1}{c}{\Large{\bf Bulk universality of sparse random matrices}}\\
\\
\\
\end{tabular}
\begin{tabular}{c c c}
Jiaoyang Huang & Benjamin Landon & Horng-Tzer Yau\\
\\
 \multicolumn{3}{c}{ \small{Department of Mathematics} } \\
 \multicolumn{3}{c}{ \small{Harvard University} } \\
\small{jiaoyang@math.harvard.edu} & \small{landon@math.harvard.edu} & \small{htyau@math.harvard.edu}  \\
\\
\end{tabular}
\\
\begin{tabular}{c}
\multicolumn{1}{c}{\today}\\
\\
\end{tabular}

\begin{tabular}{p{15 cm}}
\small{{\bf Abstract:} We consider the adjacency matrix of the ensemble of Erd\H{o}s-R\'enyi random graphs which consists of graphs on $N$ vertices in which each edge occurs independently with probability $p$.  We prove that in the regime $pN \gg 1$ these matrices exhibit bulk universality in the sense that both the averaged $n$-point correlation functions and distribution of a single eigenvalue gap coincide with those of the GOE.  Our methods extend to a class of random matrices which includes sparse ensembles whose entries have different variances.}
\end{tabular}
\end{table}

\section{Introduction}

{\let\thefootnote\relax\footnote{The work of B.L. is partially supported by NSERC.  The work of H.-T. Y. is partially supported by NSF Grant DMS-1307444 and a Simons Investigators fellowship.}}The universality of the spectral statistics of random matrices has been a central subject since the pioneering works of Wigner \cite{wigner}, Gaudin \cite{gaudin}, Mehta \cite{mehta} and Dyson \cite{dyson}.  The first such universality result is the global semicircle law of Wigner which states that under some weak moment conditions, the empirical eigenvalue distribution of a matrix with i.i.d. entries converges weakly to the deterministic semicircle law
\begin{align}
\lim_{N \to \infty} \frac{1}{N} \sum_{i=1}^N \delta_{\lambda_i } \to \rhosc (E) := \frac{1}{2 \pi } 1_{ \{ |E| \leq 2 \} } \sqrt{ 4 - E^2 }
\end{align}
in the appropriate scaling.  

The Wigner-Dyson-Gaudin-Mehta conjecture, or `bulk universality' conjecture, states that the local statistics of the eigenvalues of random matrix ensembles should be universal in the sense that they depend only on the symmetry class of the random matrix ensemble but are otherwise independent of the law of the matrix entries.  Here, local statistics refers to the behaviour of the eigenvalues in the scaling in which their typical distance is order $1$.  

A prominent class of random matrices are Wigner matrices. These matrices have independent centered entries with a uniform subexponential decay condition and identical variances.  The Wigner-Dyson-Gaudin-Mehta conjecture for Wigner matrices was recently established in a series of papers \cite{erdos2010bulk, EKYYb, erdos2011universality, erdos2012rigidity, general, gap} for all symmetry classes.  Parallel results in various cases were obtained in \cite{tao2011random, tao2010random}.

The conclusion of the papers  \cite{erdos2010bulk, EKYYb, erdos2011universality, erdos2012rigidity, general, gap} was that Wigner matrices, and even the wider class of generalized Wigner matrices in which the variances of the entries may differ, exhibit bulk universality in the following two forms.  The first is that the $n$-point correlation functions are universal after averaging over a small energy window.  The second is that the distribution of the eigenvalue gaps with a fixed label are universal.  For Wigner matrices, the universality of the averaged $n$-point correlaton functions is equivalent to the universality of a local average of eigenvalue gaps.  However, there is no rigorous mathematical relation between the universality of the eigenvalue gaps with a fixed label and the universality of the $n$-point correlation functions at a fixed energy.  Universality at a fixed energy has recently been established for all symmetry classes in \cite{Homog}, but we will not be concerned with this type of convergence in this work.  

Wigner matrices were orginally introduced in \cite{wigner} by Wigner to model the spectra of heavy atoms, and are widely used to model systems in which all elements strongly interact with one another. However, for systems in which the links between different elements are broken, a better description is offered by the so-called sparse (or dilute) random matrices which have an average of $pN$ nonzero elements per row, for $p\ll 1$.  

Aside from theoretical physics models, sparse random matrices also arise in graph theory in the study of sparse random graphs. Perhaps the simplest example is the Erd\H{o}s-R\'enyi ensemble which consists of a random graph on $N$ vertices in which each edge is chosen independently with probability $p$. The adjacency matrix of this graph is called the Erd\H{os}-R\'enyi matrix.  The Erd\H{o}s-R\'enyi matrix has typically $pN$ nonzero entries in each column and  is sparse if $p \ll 1$. As the matrix entries take values $0$ or $1$, the mean of the entries is not $0$.  Ignoring the nonzero mean for the moment, the Erd\H{o}s-R\'enyi matrix can be viewed as a singular Wigner matrix, as the probability distribution of the matrix elements is highly concentrated around $0$.  The singular nature of this ensemble can be expressed by the fact that the $k$-th moment of a matrix entry is bounded by
\begin{align}
N^{-1} (p N)^{- (k-2)/2}.
\end{align}
When $p \ll 1$, this decay in $k$ is much slower than in the case of Wigner matrices.

It was conjectured in \cite{EvEc} 
that for sparse random matrices there exists a critical value $p_c>1$, such that for $pN> p_c$, the bulk eigenvalues are strongly correlated and are characterized by GOE/GUE random matrix statistics; for $pN<p_c$, the eigenvalues remain uncorrelated and follow Poisson statistics. This conjecture is supported by a wealth of numerical simulations \cite{EvEc,Ku} and a nonrigorous supersymmetric approach \cite{FyMi,MiFy}. 
%
The best rigorous result in this direction was obtained in the works \cite{EKYYa, EKYYb} and asserts that if
\begin{align}
pN \geq N^{2/3+ \eps} \label{eqn:p23},
\end{align}
then the averaged $n$-point correlation functions of the Erd\H{o}s-R\'enyi ensemble coincide with the GOE.

In the present work we prove that in the regime
\begin{align}
pN \geq N^{\eps} \label{eqn:p1},
\end{align}
the local statistics of the Erd\H{o}s-R\'enyi ensemble exhibit bulk universality.  In addition to proving that the averaged $n$-point correlation functions coincide with the GOE, we also prove the universality of the eigenvalue gaps with a fixed label.  To further place the present work in context we recall the three-step strategy developed in  \cite{erdos2010bulk, EKYYb, erdos2011universality, erdos2012rigidity, general, gap} for proving universality for Wigner matrices:
\begin{enumerate}[label={(\arabic*)}]
\item \label{item:local} Establish a local semicircle law controlling the number of eigenvalues in windows of size $\log(N)^C/N$.
\item \label{item:dbm} Analyze the local ergodicity of Dyson Brownian motion (DBM) to obtain universality for Wigner ensembles with a small Gaussian component.
\item \label{item:pert} A density argument comparing a general Wigner matrix to one with a small Gaussian component.
\end{enumerate}
For an overview of this three-step strategy and more details we refer the reader to \cite{localspectral}.  The local semicircle law for sparse random matrices in the regime \eqref{eqn:p1} was established in \cite{EKYYa}.  However,  in \cite{EKYYb} Steps \ref{item:dbm} and \ref{item:pert} were only completed for sparse random matrices in the regime \eqref{eqn:p23}.  

The key input from Step \ref{item:local} into Step \ref{item:dbm} is a high-probability a-priori bound on the eigenvalue locations which is a corollary of the strong semicircle law.  In the case of Wigner matrices, this bound is optimal and it allows one to conclude that local equilibrium is reached by DBM in times $t = N^\eps/N$.  Sparse random matrices do not obey as strong a semicircle law and so the time to equilibrium found in the work \cite{EKYYb} was much longer.  Moreover, due to the slow decay of the third moment, the approximation in Step \ref{item:pert} is not as strong in the case of Wigner matrices and so could not be used for the large times required by Step \ref{item:dbm}.  These two factors led to the condition \eqref{eqn:p23} of \cite{EKYYb}. 

In the recent work \cite{ly}, the optimal time of Dyson Brownian motion to local equilibrium was established for a wide class of initial data (see \cite{Kevin3} for related results on DBM with general initial data).  Using this as an input we will prove that DBM reaches local equilibrium in  the optimal time $t = N^\eps / N$ when the initial data is a sparse random matrix. 
For the comparison of correlation functions,  Step \ref{item:pert} was obtained in \cite{EKYYb} via  a Green function comparison theorem. In this paper, we will use a lemma of \cite{BoYa} which asserts continuity of DBM when viewed as a matrix Ornstein-Uhlenbeck process. 
It is interesting to note that this continuity lemma provides 
a very convenient tool for Step \ref{item:pert} in the sparse setting whenever a ``weak local semicircle law" is valid -- and in the case of sparse random matrices this is provided by a result of \cite{EKYYa}. 


 The universality of a single gap was established in \cite{gap} for Wigner matrices, i.e., for $p \sim O(1)$. The work \cite{ly} also yields gap universality for DBM after the optimal time $t=N^{\eps-1}$ and so our task is similar to the proof of the correlation function universality in that we must establish Step \ref{item:pert} and compare the gap distributions.
However, the  completion of  Step \ref{item:pert} presents a major difficulty. 
Previously, for gap universality this  step was  based on results of \cite{tao2010random, tao2011random} or \cite{knowles2013eigenvector}
which states that the gap distribution of two Wigner ensembles coincide provided that  the first four moments of these two ensembles match. 
However, these results were based on two inputs: firstly, certain  level repulsion  estimates; secondly, an optimal eigenvalue rigidity estimate.


Optimal eigenvalue rigidity estimates for Wigner ensembles were proven in \cite{erdos2010universality, erdos2012rigidity}.  This estimate states that for any eigenvalue $\lambda_i$ in the bulk we have that $|\lambda_i-\gamma_i|\leq N^{-1+\delta}$ with overwhelming probability, where $\gamma_i$ is the deterministic classical location of the $i$-th eigenvalue. The best known rigidity result for sparse random matrices is from \cite{EKYYa}, where it was shown that the bulk eigenvalues satisfy $|\lambda_i-\gamma_i|\leq p^{-1}N^{-1+\delta}$ with overwhelming probability. 

Moreover, we do not expect that optimal rigidity holds for sparse random matrices. In fact in \cite{ShTi}, it was shown that for sparse random matrices the linear statistics 
\begin{align*}
N^{-1}\sum_{i=1}^{N}\phi(\lambda_i) \to \mathcal{N}
\end{align*}
converges to a normal random variable with variance $O((N^2p)^{-1})$, for $\phi$ satisfying some regularity conditions. This implies that the fluctuations of the eigenvalues are at least of order $O((N\sqrt{p})^{-1})$ on average, and so we do not expect optimal rigidity to hold if $p\ll 1$.

As mentioned above, the lack of rigidity for sparse ensembles resulted in the longer time to equilibrium for DBM being found in \cite{EKYYb}, and it again causes difficulty in trying to compare gap statistics. Rigidity results are a crucial input in establishing level repulsion estimates for Wigner matrices which are needed in order to compare the gap statistics of two ensembles. 
It was proven in \cite{ly} that a level repulsion estimate will hold for DBM after a short time.  We show that one can combine the delocalization of eigenvectors together with the Ornstein-Uhlenbeck continuity lemma of  \cite{BoYa} to pass this level repulsion from DBM to the initial sparse random matrix. This level repulsion estimate then gives us a key input for Step \ref{item:pert} and we are able to conclude universality of the gap statistics.

Previous level repulsion estimates were obtained in \cite{wegner, Homog} for Wigner ensembles whose entries have a smooth distribution.  Estimates without a smoothness condition were obtained in \cite{tao2010random, tao2011random} and also in the very recent work \cite{taolevel}.  A weak level repulsion estimate for Wigner matrices also follows from the results of \cite{gap}.

In fact, our strategy outlined above applies to a wider class of random matrices than sparse or Wigner random matrices alone.  We will prove that bulk universality holds for a class of random matrices obeying only a weak estimate on the distribution of its eigenvalues, a weak decay condition on the third moment of the entries and an eigenvector delocalization estimate. 


The remainder of the paper is outlined as follows.  In Section \ref{def} we introduce the random matrix models under consideration, which we will call `stable' random matrices, and state our main results.  In Section \ref{bulkHt} we obtain bulk universality for Gaussian divisible ensembles.  In Section \ref{levelRp} we state and prove our level repulsion results for stable random matrices.  In Section \ref{bulkH} we complete Step \ref{item:pert} outlined above and compare the bulk statistics of a general stable random matrix and a Gaussian divisible ensemble.  In Section \ref{example} we prove that sparse random matrices are stable and conclude universality for sparse random matrices.

\section{Definition of model and main results}\label{def}

In our paper we will only state and prove our results for real symmetric random matrix ensembles.  All of our methods extend with only notational changes to complex Hermitian ensembles.

\subsection{Sparse random matrices}

In this section we introduce the class of sparse random matrices that we study.  We follow the notations and definitions of \cite{EKYYa, EKYYb}.  The motivating example is the Erd\H{o}s-R\'enyi matrix whose entries are independent up to the constraint that the matrix is symmetric, and equal to $1$ with probability $p$ and $0$ with probability $1-p$.  It is notationally convenient to replace the parameter $p$ with $q$ defined through
\begin{align}
q:=  \sqrt{Np}.
\end{align}
We allow $q$ to depend on $N$.  We also rescale the matrix so that the bulk of its spectrum lies in an interval of order $1$.  For the Erd\H{o}s-R\'enyi matrix we define $H$ to be the $N\times N$ symmetric matrix whose entries $h_{ij}$ are independent up $h_{ij} = h_{ji}$ and each element is distributed according to
\begin{align}
h_{ij} = \frac{\gamma}{q} \begin{cases} 1 & \mbox{ with probability } \frac{q^2}{N} \\ 0 & \mbox{ with probability } 1 - \frac{q^2}{N} \end{cases},
\end{align}
where we have defined
\begin{align}
\gamma := \left( 1- \frac{q^2}{N} \right)^{-1/2}.
\end{align}
We further extract the mean of each entry and write
\begin{align}
H = \M + \gamma q \vert \ev \rangle \langle \ev \vert
\end{align}
where $\ev$ is the unit vector
\begin{align}
\ev := \frac{1}{ \sqrt{N}} ( 1, ..., 1 )^T.
\end{align}
Note that the matrix elements of $\M$ are centered. It is easy to check that the matrix elements of $\M$ satisfy the moment bounds 
\begin{align}
\EE [\m_{ij}^2] = \frac{1}{N} , \qquad \EE [ | \m_{ij} |^k ] \leq \frac{1}{ Nq^{k-2} },\quad k\geq 2.
\end{align}
We are prompted to make the following definition.  We introduce two parameters $q$ and $f$ which may be $N$-dependent.  

\begin{definition}[Sparse random matrices] \label{def:snrm} $H$ is a sparse random matrix with sparsity parameter $q$ and mean $f$ if it is of the form
\begin{align}
H = \M + f \vert \ev \rangle \langle \ev \vert
\end{align}
where $f$ is a deterministic number satisfying
\begin{align}
0 \leq f \leq N^{1/2}
\end{align}
and $\M$ is a matrix with real and independent  entries up to the symmetry constraint $\m_{ij} = \m_{ji}$ which satisfy
\begin{align} \label{momentbound}
\EE [ \m_{ij} ] = 0, \quad \EE [ | \m_{ij} |^2 ] = \frac{1}{N} , \quad \EE [ |\m_{ij} |^k ] \leq \frac{ C^k } { N q^{k-2} }
\end{align}
for $1 \leq i \leq j \leq N$ and $2 \leq k \leq \log(N)^{ 10 \log \log N }$ where $C$ is a positive constant. We assume that $q$ satisfies
\begin{align}
N^{\alpha} \leq q \leq N^{1/2}
\end{align}
for some $\alpha >0$.

\end{definition}

\subsection{Universality of sparse random matrices}

Our main result is the bulk universality of sparse random matrices as defined above.

\begin{theorem}\label{bsparse}
 Let $H$ be a sparse random matrix as defined in Definition \ref{def:snrm}, with sparsity parameter $q$ satisfying
 \begin{align}
 N^{\alpha} \leq q \leq N^{1/2},
 \end{align}
for some number $\alpha>0$. Then $H$ exhibits bulk universality in the following two forms. Firstly, $H$ has the single gap universality in the bulk. For any $\kappa>0$ and index $i\in [[\kappa N, (1-\kappa)N]]$
\begin{align}
\begin{split}
\lim_{N\rightarrow \infty}&\EE^{(H)}[ O(N \rho_{sc}(\gamma_i)(\lambda_i-\lambda_{i+1}), \cdots, N\rho_{sc}(\gamma_i)(\lambda_i -\lambda_{i+n}))]\\
-&\EE^{(GOE)}[ O(N \rho_{sc}(\gamma_i)(\lambda_i-\lambda_{i+1}), \cdots, N\rho_{sc}(\gamma_i)(\lambda_i -\lambda_{i+n}))]=0.
\end{split}
\end{align}
Secondly, the averaged $n$-point correlation functions of $H$ are universal in the bulk. We denote the $n$-point correlation function functions of $H$ and $GOE$ by $\rho_H^{(n)}$ and $\rho_{GOE}^{(n)}$ respectively, then for any $\delta >0$ and $E\in (-2, 2)$,  and $b\geq N^{-1+\delta}$ 
\begin{align}
\begin{split}
\lim_{N\rightarrow \infty}\int_{E-b}^{E+b} \int_{\RR^n} O(\alpha_1,\cdots, \alpha_n)\left\{ \frac{1}{\rho_{sc}(E)^n}\rho^{(n)}_{H}\left( E'+\frac{\alpha_1}{N\rho_{sc}(E)},\cdots E'+\frac{\alpha_n}{N\rho_{sc}(E)}\right)\right.\\
-\left.\frac{1}{\rho_{sc}(E)^n}\rho^{(n)}_{GOE}\left( E'+\frac{\alpha_1}{N\rho_{sc}(E)},\cdots E'+\frac{\alpha_n}{N\rho_{sc}(E)}\right)\right\} \d \alpha_1 ... \d \alpha_n \frac{\d E'}{2b}=0.
\end{split}
\end{align}
where the test function $O\in C^{\infty}_{c}(\RR^n)$.
\end{theorem}

\subsection{Stable random matrices}\label{def23}
While our main goal is to study the sparse random matrices defined above, we note that our analysis applies to a somewhat more general class of random matrices.  We consider an $N\times N$ real symmetric random matrix $H=(h_{ij})_{1\leq i,j\leq N}$, which satisfy $\EE[h_{ij}]=\f$ and $\EE[(h_{ij}-\f)^2]=s_{ij}$.  We assume that there are constants $c_1$ and $c_2$ such that
\begin{align}c_1N^{-1}\leq s_{ij}\leq c_2N^{-1},
\end{align}
and the mean $0\leq \f \leq N^C$ may depend on $N$.

We define the following matrix stochastic differential equation \cite{BoYa} which is an Ornstein-Uhlenbeck version of the Dyson Brownian motion. The dynamics of the matrix entries are given by the stochastic differential equations
\begin{align}
\label{DBM} \d\left(h_{ij}(t)-\f\right) =\frac{\d B_{ij}(t)}{\sqrt{N}}-\frac{1}{2Ns_{ij}}\left(h_{ij}(t)-\f \right)\d t.
\end{align}
where $B$ is symmetric with $(B_{ij}(t))_{1\leq i\leq j\leq N}$ a family of independent Brownian motions. We denote $H_t= (h_{ij}(t))_{1\leq i,j\leq N}$, and so $H_0=H$ is our original matrix. More explicitly, for the entries of $H_t$, we have
\begin{align}
\label{fDBM} h_{ij}(t)=\f+\e^{-\frac{t}{2Ns_{ij}}}\left(h_{ij}(0)-\f\right)+\frac{1}{\sqrt{N}}\int_{0}^{t}\e^{\frac{s-t}{2Ns_{ij}}}\d{B_{ij}(s)}.
\end{align}
Clearly, for any $t\geq 0$, we have $\EE[h_{ij}(t)]=\f$, and $\EE[\left(h_{ij}(t)-\f\right)^2]=s_{ij}$. More importantly, the law of $h_{ij}(t)$ is Gaussian divisible, i.e. it contains a copy of Gaussian random variable with variance $O(tN^{-1})$. Therefore $H_t$ can be written as 
\begin{align}
\label{modOU}H_t\stackrel{d}{=}H_t^{(1)}+\sqrt{\frac{r(1-e^{-t/r})}{2}}G,
\end{align}
where $r=\min_{i\leq j}\{Ns_{ij}\}$, $G$ denotes a standard gaussian orthogonal ensemble, which is independent of $H_t^{(1)}$. The entries of the matrix $H_t^{(1)}$ is given by
\begin{align*}
\left(H_t^{(1)}\right)_{ij}\stackrel{d}{=}
\f+\e^{-\frac{t}{2Ns_{ij}}}\left(h_{ij}(0)-\f\right)+\sqrt{Ns_{ij}(1-e^{-\frac{t}{Ns_{ij}}})-\frac{1+\delta_{ij}}{2}r(1-e^{-\frac{t}{r}})}\frac{\tilde{B}_{ij}(t)}{\sqrt{N}},
\end{align*}
where $\tilde{B}$ is symmetric with $(\tilde{B}_{ij}(t))_{1\leq i\leq j\leq N}$ a family of independent Brownian motions.

We define the deformed matrix $\theta^{ab}H_t$ by
\begin{align}
(\theta^{ab}H_t)_{kl}:=\f+\theta^{ab}_{kl}\left(h_{kl}(t)-\f\right),
\end{align}
where $\theta_{kl}^{ab}=1$ unless $\{k,l\}=\{a,b\}$ in which case $\theta_{ab}^{ab} = \theta_{ba}^{ab}$ will be a number satisfying $0\leq \theta_{ab}^{ab}=\theta_{ba}^{ab}\leq 1$. 

\begin{definition}\label{general}
Let $A$ be an $N\times N$ deterministic real symmetric matrix. We denote the eigenvalues of $A$ as $\{\lambda_1,\lambda_2,\cdots, \lambda_N\} $ and corresponding eigenvectors $\{u_1,u_2,\cdots, u_N\}$. For any (small) number $\delta>0$, we call the matrix $M$ $\delta$-general if: 
\begin{enumerate}[label={(\arabic*)}]
\item The eigenvectors of $A$ are completely delocalized: $\sup_{i, j} |u_i(j)|^2\leq CN^{-1+\delta}$.
\item The eigenvalues of $A$ do not accumulate:  there is an universal constant $C$, such that for any interval $I$ with length $|I|\geq N^{-1+\delta}$, we have $\#\{i: \lambda_i\in I\}\leq C|I|N$. 
\end{enumerate} 
\end{definition}

\begin{definition}\label{stabledef}
We call the random matrix $H$ {\it{stable}} if: 
\begin{enumerate}[label={(\arabic*)}]
\item The entries of $H$ are independent up to symmetry.\label{ind}
\item For any time $t=N^{-1+\epsilon}$, where $\epsilon>0$ can be arbitrarily small, the random matrix $H_t^{(1)}$ defined in \eqref{modOU} satisfies the weak local semicircle law, i.e. for any (large) number $D>0$, and (small) number $\delta>0$, the following holds with probability larger than $1-N^{-D}$,
\begin{align}
\label{loclawHt1}\left|\frac{1}{N}\Tr (H_t^{(1)}-E-i\eta)^{-1}-m_{sc}(E+i\eta)\right|\leq N^{-c(\delta)}, 
\end{align}
uniformly for $-5\leq E\leq 5, N^{-1+\delta}\leq \eta\leq 10$,
where $c(\delta)>0$ is some constant depending on $\delta$, and $\spc (H_t^{(1)} ) \subseteq (-3, N^C)$ for some fixed $C$. \label{loclaw}

\item There exists some universal constant $\alpha$, such that $\EE[|h_{ij}-\f|^3]\leq N^{-1-\alpha}$.\label{moment}
\item For any time $0\leq s\leq t$, any (large) number $D>0$, and (small) number $\delta>0$, $\theta^{ab}H_{s}$ is $\delta$-general with probability larger than $1-N^{-D}$, with the constants in Definition \ref{general} uniformly in $s$.  \label{stable}
\end{enumerate}
\end{definition}

\begin{remark}
The second condition \ref{loclaw} implies that for any $\kappa >0$, the eigenvalues $\lambda_i ( H_t^{(1)} ) \in (-2 + \kappa^{2/3}, 2 - \kappa^{2/3})$ for $i \in [[2\kappa N, (1 - 2\kappa ) N ]]$ with overwhelming probability.
\end{remark}
\begin{remark}
In order to simplify our proof we have assumed that the matrix elements are independent. Independence is mainly used in the comparison Lemma \ref{compare}, which will still hold if the matrix entries of $H$ are weakly correlated. 
\end{remark}

\begin{remark}
 The motivating example of our paper is the sparse random matrix, and we have therefore assumed that the Stieltjes transform of the empirical eigenvalue distribution of $H_t^{(1)}$ is close to $m_{sc}$, i.e., the semicircle law. The semicircle law, however, does not play an active role 
and   our methods can be applied to the case in which the semicircle law is replaced by other  densities. We will not  pursue this direction and  refer the  interested reader to \cite{ajanki1, ajanki2} and \cite {adlam}  for examples in which the limiting eigenvalue density differs from the semicircle law. 
\end{remark}


In this paper we will prove that the local statistics of stable random matrices are universal.

\begin{theorem}{\label{TbulkH}}
Let $H$ be a stable random matrix as defined in Definition \ref{stabledef}. The local statistics of $H$ in the bulk are universal. Firstly, $H$ has gap universality with a fixed label in the bulk. For any $\kappa>0$ and index $i\in [[\kappa N, (1-\kappa)N]]$, we have
\begin{align}\label{Hgapu}
\begin{split}
\lim_{N\rightarrow \infty}&\EE^{(H)}[ O(N \rho_{sc}(\gamma_i)(\lambda_i-\lambda_{i+1}), \cdots, N\rho_{sc}(\gamma_i)(\lambda_i -\lambda_{i+n}))]\\
-&\EE^{(GOE)}[ O(N \rho_{sc}(\gamma_i)(\lambda_i-\lambda_{i+1}), \cdots, N\rho_{sc}(\gamma_i)(\lambda_i -\lambda_{i+n}))]=0.
\end{split}
\end{align}
The averaged $n$-point correlation functions of $H$ are universal in the bulk. We denote the $n$-point correlation function functions of $H$ and $GOE$ by $\rho_H^{(n)}$ and $\rho_{GOE}^{(n)}$ respectively, then for any $\delta >0$ and $E\in (-2, 2)$,  and $b\geq N^{-1+\delta}$ 
\begin{align}\label{Hcoru}
\begin{split}
\lim_{N\rightarrow \infty}\int_{E-b}^{E+b} \int_{\RR^n} O(\alpha_1,\cdots, \alpha_n)\left\{ \frac{1}{\rho_{sc}(E)^n}\rho^{(n)}_{H}\left( E'+\frac{\alpha_1}{N\rho_{sc}(E)},\cdots E'+\frac{\alpha_n}{N\rho_{sc}(E)}\right)\right.\\
-\left.\frac{1}{\rho_{sc}(E)^n}\rho^{(n)}_{GOE}\left( E'+\frac{\alpha_1}{N\rho_{sc}(E)},\cdots E'+\frac{\alpha_n}{N\rho_{sc}(E)}\right)\right\}\d \alpha_1 ... \d \alpha_n \frac{dE'}{2b}=0.
\end{split}
\end{align}
Above, the observable $O\in C^{\infty}_{c}(\RR^n)$.
\end{theorem}

\section{Bulk universality of $H_t$}{\label{bulkHt}}

The goal of this section is to establish bulk universality for the matrix valued stochastic process $H_t$ defined as in \eqref{modOU} after a short time $t = N^{-1+\eps}$. 

\begin{theorem} \label{thm:htbulk}
Let $H$ be a stable random matrix, and let $H_t$ be defined as in \eqref{modOU}.  For any small $\eps,\kappa >0$ there is a constant $c >0$, which depends on $\eps$, such that the following holds for $t = N^{-1 + \eps}$. and any index $i \in [[\kappa N, ( 1 - \kappa ) N ]]$
\begin{align} \label{eqn:htgap}
\bigg| &\EE^{(H_t)} [ O ( N ( \lambda_i - \lambda_{i+1} ) , \cdots , N ( \lambda_i - \lambda_{i + n } ) ) ]\notag \\
 -   &\EE^{(GOE)} [ O ( N ( \lambda_i - \lambda_{i+1} ) , \cdots , N ( \lambda_i - \lambda_{i + n } ) ) ] \bigg| \leq N^{-c}
\end{align}
Moreover, for any $\delta,\kappa>0$, $E\in (-2+\kappa,2-\kappa)$ and $b\geq N^{-1+\delta}$, we have
\begin{align} \label{eqn:htbulk}
\big| &\int_{E - b}^{E+b} \int_{\RR^n} O ( \alpha_1, ..., \alpha_n ) \bigg\{ \frac{1}{ \rhosc (E)^n } \rho_{H_t}^{(n)} \left( E' + \frac{ \alpha_1}{ N \rhosc (E) } , ..., E' + \frac{ \alpha_n}{ N \rhosc (E) } \right) \notag \\
- & \frac{1}{\rhosc (E)^n } \rho_{GOE}^{(n)} \left( E' + \frac{ \alpha_1}{ N \rhosc (E) } , ..., E' + \frac{ \alpha_n } { N \rhosc (E) } \right) \bigg\} \d \alpha_1 ... \d \alpha_n \frac{ \d E'}{b} \big| \leq N^{ - c}
\end{align}
where the test function $O\in C_c^{\infty}(\RR^n)$.
\end{theorem}

For the proof we shall first restate the main result of \cite{ly} in a form convenient for our proof.  For this we will introduce some notation.  For a deterministic matrix $A$ we define
\begin{align}
A_t := A + \thetat G,\quad \thetat := \sqrt{\frac{r(1-e^{-t/r})}{2}}=O(t^{\frac{1}{2}}).
\end{align}
where $G$ is a GOE matrix and $r$ is the constant from \eqref{modOU}. We denote by $m_t$ the  Stieltjes transform of the free convolution of the empirical eigenvalue distribution of $A$ and the semicircle law of $\thetat G$, and so $m_0$ is the Stieltjes transform of empirical eigenvalue distribution of $A$.  More explicitly, $m_t$ is defined as the unique solution to the functional equation
\begin{align}
m_t (z) := m_{0}(z+\thetat^2 m_t(z))=\frac{1}{N} \sum_{i=1}^N \frac{1}{ \lambda_i (A) - z - \thetat^2 m_t (z) }, \quad \mathrm{Im} [ m_t ] \geq 0, \mathrm{Im} [z] \geq 0. \label{eqn:Acon}
\end{align}
The free convolution is well-studied, see, e.g, \cite{biane}.  It is known that $m_t$ is the Stieltjes transform of a measure with a density which we denote by $\rho_t$ which is analytic on the interior of its support, for any $t>0$.  Denote the classical eigenvalue locations of the density $\rho_{sc}$ and $\rho_t$ by $\gamma_i$ and $\gamma_{i,t}$, respectively, 
\begin{align*}
\int_{-\infty}^{\gamma_i} \rho_{sc}(x)dx=\frac{i}{N},\quad
\int_{-\infty}^{\gamma_{i,t}} \rho_{t}(x)dx=\frac{i}{N}.
\end{align*}
The following follows from Theorem 2.5 of \cite{ly}.
\begin{theorem}  \label{thm:abulk} Fix $\eps >0$ and suppose that there are constants $c_1 >0$ and $C_1 > 0$ such that
\begin{align}
c_1 \leq \mathrm{Im } [ m_{0} (E + \i \eta ) ] \leq C_1
\end{align}
for all $E \in (-2 + \kappa, 2 - \kappa )$ and $N^{-1+\eps/3} \leq \eta \leq 10$.  Suppose furthermore that $\spc (A) \subseteq [-C, N^C)$ for some fixed $C$.  Let $i$ be such that $\lambda_i (A) \in (-2 + 2 \kappa, 2 - 2 \kappa)$. Then for $t=N^{-1+\epsilon}$, there exists a small constant $c>0$, which depends on $\epsilon$, such that for indices $i_1, ..., i_n \in \nn$ and $i_k \leq N^{c}$,
\begin{align}
\big| &\EE^{(A_t)} [ O (  \rho_t ( \gamma_{i,t} ) N ( \lambda_i - \lambda_{i+i_1} ) , \cdots ,  \rho_t ( \gamma_{i,t} )N ( \lambda_i - \lambda_{i + i_n } ) ) ]\notag \\
 -   &\EE^{(GOE)} [ O ( \rhosc (\gamma_j )N ( \lambda_j - \lambda_{j+i_1} ) , \cdots ,  \rhosc ( \gamma_j )N ( \lambda_j - \lambda_{j + i_n } ) ) ] \big| \leq N^{-c}
\end{align}
for all large enough $N\geq N(\epsilon)$. Above $j$ is any index satisfying $j \in [[ \kappa_1 N , (1 - \kappa_1 ) N ]]$ where $\kappa_1 >0$.

\end{theorem}

We will apply the above theorem by conditioning on $H_t^{(1)}$ and taking $A = H_t^{(1)}$.  However, we must take care of the scaling by $\rho_t ( \gamma_{i,t})$ in the above theorem statement.  This becomes a random variable depending on $H_t^{(1)}$.  We want to replace it by the deterministic quantity $\rhosc ( \gamma_i )$; we will see that our assumption on the weak local law of $H_t^{(1)}$ combined with continuity of the free convolution allows us to do this.

Define the spectral domain 
\begin{equation}
\F=\{z=E+i\eta: E \in (-5, 5), N^{-1+\eps/3} \leq \eta \leq 10\},
\end{equation}
 and the set of real symmetric $N\times N$ matrices
\begin{align}
\label{defD} \D &:= \{ A : |m_{0 } ( z ) - \msc (z ) | \leq  N^{ - \omega }, z\in \F  \} \cap \{ \spc (A) \subseteq (-3, N^C) \}.
\end{align}
where $\omega$ satisfies 
\begin{equation}
0<\omega\leq \min\{\epsilon/3,c(\epsilon/3), (1-\epsilon)/4\},
\end{equation}
 and $c(\epsilon/3)$ is from the constant in \eqref{loclawHt1} of Definition \ref{stabledef}.

\begin{lemma} \label{lem:mbds}
Let $\kappa,\omega>0$ and $A \in \D$ as above.  Then, for any $\epsilon>0$ and $t = N^{-1+\eps}$,
\begin{align}
| m_{t} (z) - \msc (z) | \leq CN^{ -\omega} ,  
\end{align}
uniformly for any $z \in \{ z = E + \i \eta : E \in (-4, 4) , N^{-1+\epsilon/3} \leq \eta \leq 9 \}$.
\end{lemma}
\begin{proof}
From Jensen's inequality we have:
\begin{align*}
|m_{t}|^{2}\leq \frac{1}{N}\sum \frac{1}{|\lambda_i(A)-z-\thetat^2m_{t}(z)|^2}=\frac{\text{Im} [m_{t}(z)]}{\text{Im} [z+\thetat^2 m_{t}(z)]}\leq \thetat^{-2}.
\end{align*}
Therefore we have $|\thetat^2 m_{t}|\leq \thetat=O(t^{1/2})$. Since $A\in \D$, for any $z=E+i\eta$, such that $E\in (-4,4)$ and $N^{-1+\epsilon/3}\leq \eta\leq 9 $, we have that $z+\thetat^2m_{t}\in \mc{F}$. The defining relation \eqref{defD} leads to
\begin{align}
\label{mbound} m_{t}(z)=m_{0}(z+\thetat^2m_{t}(z))=m_{sc}(z+\thetat^2 m_{t}(z))+O(N^{-\omega})
\end{align}
To control $m_{sc} (z+\thetat^2 m_{t})$ we have the following stability estimate of $m_{sc}$: for any $z, \Delta z$ with non-negative imaginary part, and $|\Delta z|\leq 1$, then
\begin{align}
\label{msc}|m_{sc}(z+\Delta z)-m_{sc}(z)|\leq 2|\Delta z|^{1/2}.
\end{align}
Therefore \eqref{mbound} gives us
\begin{align*}
|m_{t}(z)-m_{sc}(z)|\leq  |m_{t}(z)-m_{sc}(z+\thetat^2 m_{t})|+|m_{sc}(z+\thetat^2 m_{t})-m_{sc}(z)|\leq CN^{-\omega},
\end{align*}
given that $\omega<(1-\epsilon)/4$.
\end{proof}

From this lemma we conclude the following.  
\begin{lemma} \label{lem:determbds}
For any $\epsilon, \kappa>0$, time $t=N^{-1+\epsilon}$, any real symmetric matrix $A\in \D$, and $E\in (-2+\kappa, 2-\kappa)$ we have
\begin{align}
\label{detE}\left| \rho_t ( E) - \rhosc ( E ) \right| \leq C N^{- \omega}.
\end{align}
Moreover for any index $i$ such that $\lambda_i(A)\in (-2+\kappa, 2-\kappa)$, we have 
\begin{align}
\left| \rho_t ( \gamma_{i,t} ) - \rhosc ( \gamma_i ) \right| \leq C N^{- \omega/2 }.
\end{align}
where the constant $\omega$ is from the definition \eqref{defD} of the set $\D$.
\end{lemma}

\begin{proof}
From \cite[Lemma 7.1]{ly}, the derivative of $\rho_t$ satisfies $|\rho_{t}'(E)|\leq C/t$ for $E\in (-2+\kappa, 2-\kappa)$, where we use the fact $\thetat^2=O(t)$.  Moreover the same lemma also shows that $\rho_t (E) \leq C$ on $(-2 + \kappa/2, 2 - \kappa/2 )$. Given $E\in (-2+\kappa, 2-\kappa)$, we denote the intervals $I_1=[E-N^{-1+2\epsilon/3}, E+N^{-1+2\epsilon/3}]$, $I_2 = \{ x: N^{-1+2\epsilon/3} < |x - E | < \kappa/2 \}$ and $I_3 = \RR \backslash ( I_1 \cup I_2 )$. We take $\eta=N^{-1+\epsilon/3}$, then 
\begin{align}
\notag&|\rho_{t}(E)-\frac{1}{\pi} \text{Im }[m_{t}(E+i\eta)]|\\
\notag\leq& \left|\frac{1}{\pi}\int_{I_1} \frac{(\rho_{t}(E)-\rho_{t}(x))\eta \d x}{(x-E)^2+\eta^2}dx\right|+ \left|\frac{1}{\pi}\int_{I_2 \cup I_3} \frac{\rho_{t}(E)\eta \d x}{(x-E)^2+\eta^2}\right| \\
+ &\left|\frac{1}{\pi}\int_{I_2} \frac{\rho_{t}(x)\eta \d x}{(x-E)^2+\eta^2}\right| +
\left|\frac{1}{\pi}\int_{I_3} \frac{\rho_{t}(x)\eta \d x}{(x-E)^2+\eta^2}\right|\notag \\
\label{bound1}\leq& \sup_{x \in I_1}|\rho_{t}(x)-\rho_{t}(E)|+ C N^{-\eps/3} + CN^{-\eps/3} + C\frac{\eta}{ \kappa^2} \int_{I_3} \rho_t (x) \d x \leq CN^{-\epsilon/3},
\end{align}
where we have used $|\rho_t(x) | \leq C$ on $I_1\cup I_2$.
Moreover, we have
\begin{align}
\notag&\left|\frac{1}{\pi} \text{Im }[m_{t}(E+i\eta)]-\rho_{sc}(E)\right| \\
\notag\leq& 
\frac{1}{\pi} \left|\text{Im }[m_{t}(E+i\eta)]-\text{Im }[m_{sc}(E+i\eta)]\right|+\frac{1}{\pi} |\text{Im }[m_{sc}(E+i\eta)]-\frac{1}{\pi} \text{Im }[m_{sc}(E)]|\\
\label{bound2}\leq& N^{-\omega}+\frac{2}{\pi}\sqrt{\eta}\leq CN^{-\omega},
\end{align}
given $\omega \leq (1-\epsilon)/2$. \eqref{bound1} and \eqref{bound2} together lead to \eqref{detE}.

Moreover, by our hypothesis that $A\in \D$, $\spc (A)$ is bounded below by $3$. Therefore the density $\rho_t$ is also bounded below: $\text{supp }\rho_t \subseteq [-7/2, \infty)$. Therefore if $i$ is such that $\lambda_i (A) \in (-2 + \kappa, 2 - \kappa)$ then 
\begin{align}
| \gamma_{i,t} - \gamma_i | \leq N^{ - \omega/2 }.
\end{align}
This follows from \cite[Lemma 7.17]{ly}, using Lemma \ref{lem:mbds} as input.  Therefore,
\begin{align}
\left| \rho_t ( \gamma_{i,t} ) - \rhosc ( \gamma_i ) \right| \leq \left| \rho_t ( \gamma_{i,t} ) - \rhosc ( \gamma_{i,t} ) \right| +\left| \rhosc ( \gamma_{i,t} ) - \rhosc ( \gamma_i ) \right| \leq C N^{ - \omega/2}.
\end{align}
\end{proof}

\subsection{Proof of Theorem \ref{thm:htbulk}}

By the hypotheses of stability of $H$ we have that $H_t^{(1)} \in \D$ with probability greater than $1 - N^{-D}$ for any large $D$. We denote the density of free convolution of $H_t^{(1)}$ and $\thetat G$ as $\rho_t^H$, and its $i$-th classical eigenvalue location as $\gamma_{i,t}^{H}$. From Theorem \ref{thm:abulk} we have
\begin{align}
\big| &\EE^{(H_t)} [ O (  \rho_t^{H} (\gamma_{i,t}^H ) N ( \lambda_i - \lambda_{i+i_1} ) , \cdots , \rho_t^{H} (\gamma_{i,t}^H ) N ( \lambda_i - \lambda_{i + i_n } ) ) | H_t^{(1)} ]\notag \\
 \label{condgap}-   &\EE^{(GOE)} [ O ( \rhosc ( \gamma_i ) N ( \lambda_i - \lambda_{i+i_1} ) , \cdots , \rhosc ( \gamma_i )N ( \lambda_i - \lambda_{i + i_n } ) ) ] \big| \leq N^{-c}.
\end{align}
But then by Lemma \ref{lem:determbds} we have
that $|\rho_t^{H}(\gamma_{i,t}^{H})-\rho_{sc}(\gamma_i)|\leq N^{-\omega/2}$, therefore
\begin{align}
\big| &\EE^{(H_t)} [ O (  \rho_t^H (\gamma_{i,t}^H ) N ( \lambda_i - \lambda_{i+i_1} ) , \cdots , \rho_t^H (\gamma_{i,t}^H) N ( \lambda_i - \lambda_{i + i_n } ) ) | H_t^{(1)} ]\notag \\
 -   &\EE^{(H_t)} [ O ( \rhosc ( \gamma_i ) N ( \lambda_i - \lambda_{i+i_1} ) , \cdots , \rhosc ( \gamma_i )N ( \lambda_i - \lambda_{i + i_n } ) ) | H_t^{(1)}] \big| \leq CN^{-\omega/2}
\end{align}
for some $C>0$ depending on first derivative and the support of the test function $O$.  We therefore obtain that for $H_t^{(1)} \in \D$ that
\begin{align} \label{eqn:htgapcond}
\big| &\EE^{(H_t)} [ O ( \rhosc ( \gamma_i ) N ( \lambda_i - \lambda_{i+i_1} ) , \cdots , \rhosc ( \gamma_i )N ( \lambda_i - \lambda_{i + i_n } ) ) | H_t^{(1)}]\notag \\
 -   &\EE^{(GOE)} [ O ( \rhosc ( \gamma_i )N ( \lambda_i - \lambda_{i+i_1} ) , \cdots , \rhosc ( \gamma_i ) N ( \lambda_i - \lambda_{i + i_n } ) ) ] \big| \leq N^{-c}
\end{align}
if we choose $c$ small enough, such that $c<\omega/2$. And if we take expectation over $H_t^{(1)}$, \eqref{eqn:htgap} follows.

Now let $\tilde{\rho}_t^{(n)}$ be the $n$-point correlation functions of $H_t$ conditioned on $H_t^{(1)}$.  It is well known that the estimate \eqref{condgap} together with the optimal rigidity of the eigenvalues of $H_t$, \cite[Theorem 3.3]{ly} implies that for $H_t^{(1)} \in \D$ we have
\begin{align}
\bigg| &\int_{E - b}^{E+b} \int_{\RR^n} O ( \alpha_1, ..., \alpha_n ) \bigg\{ \frac{1}{ \rhosc (E)^n } \tilde{\rho}_t^{(n)} \left( E' + \frac{ \alpha_1}{ N \rhosc (E) } , ..., E' + \frac{ \alpha_n}{ N \rhosc (E) } \right) \notag \\
- & \frac{1}{\rhosc (E)^n } \rho_{GOE}^{(n)} \left( E' + \frac{ \alpha_1}{ N \rhosc (E) } , ..., E' + \frac{ \alpha_n } { N \rhosc (E) } \right) \bigg\} \d \alpha_1 ... \d \alpha_n \frac{ \d E'}{b} \bigg| \leq N^{ - c}
\end{align}
for $b = N^{-1+\delta}$ for any $\delta>0$. For this argument, we refer the reader to, e.g., \cite[Theorem 2.1]{erdos2011universality}. We conclude \eqref{eqn:htbulk} by integrating over $H_t^{(1)}$. \qed

\section{Level repulsion for stable random matrices}\label{levelRp}
In this section we prove the following level repulsion estimate for stable random matrices.  It will be used for the comparison of the single gap statistics between $H$ and $H_t$ in Section \ref{bulkH}. 
\begin{theorem} \label{tlevp}
Let $H$ be a stable random matrix as defined in Section \ref{def23}. Given any $0<\tau<\alpha/8$, any (small) number $\kappa>0$, and any index $i\in [[\kappa N, (1-\kappa)N]]$, we have
\begin{align}\label{levp}
\PP(|\lambda_i(H)-\lambda_{i+1}(H)|\leq N^{-1-\tau})\leq N^{-\tau/2}.
\end{align} 
\end{theorem}
\begin{remark}
The above estimate suffices for the comparison of the single gap statistics of $H$ and $H_t$.  We have not tried to optimize the exponent $-\tau/2$, which is far from optimal. The proof below is easily modified to give $-\tau +\delta$ for any $\delta >0$.
\end{remark}

It was proven in \cite{ly} that a level repulsion estimate holds for the matrix $H_t$ for $t = N^{-1+\epsilon}$. To obtain the level repulsion estimate for $H$, we need to prove that the change of eigenvalues up to time $t=N^{-1+\epsilon}$ is negligible. For this we will repeatedly use the following lemma which asserts continuity of DBM when viewed as a matrix Ornstein-Uhlenbeck process.  It is a minor modification of \cite[Lemma A.2]{BoYa}.

\begin{lemma}{\label{compare}}
Let $H$ be an $N\times N$ real symmetric random matrix $H=(h_{ij})_{1\leq i,j\leq N}$, where the $h_{ij}$'s are independent up to symmetry constraint $h_{ij} = h_{ji}$. Suppose that for constants $c_1, c_2, c_3$ its entries satisfy $\EE[h_{ij}]=\f$ and $\EE[\left(h_{ij}-\f\right)^2]=s_{ij}$, where $c_1N^{-1}\leq s_{ij}\leq c_2N^{-1}$ and $\f\geq 0$ may depend on $N$. 

Define $H_t$ as in \eqref{DBM}. Let $F$ be a smooth function on the space of real symmetric matrices satisfying
\begin{align*}
\sup_{0\leq s\leq t, a\leq b} \sup_{ \theta^{ab}} \EE\left[(N^2|h_{ab}(s)-\f|^3+N|h_{ab}(s)-\f|)|\del_{ab}^{(3)} F(\theta^{ab}H_s)|\right]\leq B
\end{align*}
Above, the deformed matrix $\theta^{ab} H_s$ is defined by $(\theta^{ab}H)_{kl}=f+\theta^{ab}_{kl}\left(h_{kl}-f\right)$, where $\theta_{kl}^{ab}=1$ unless $\{k,l\}=\{a,b\}$ and $\theta_{ab}^{ab} = \theta_{ba}^{ab}$ is a number satisfying $0\leq \theta_{ab}^{ab}=\theta_{ba}^{ab}\leq 1$. Then
\begin{align}{\label{diff}}
|\EE[F(H_t)]-\EE[F(H_0)]|\leq C tB.
\end{align}
\end{lemma}

Given a real symmetric matrix $A$, we denote its eigenvalues by $\{\lambda_{1}, \lambda_{2},\cdots, \lambda_{N}\}$, and corresponding eigenvectors $\{u_1, u_2,\cdots, u_N\}$. If $\lambda_i$ is a simple eigenvalue of $A$, we define $P_i$ to be the orthogonal projection to the one-dimensional eigenspace corresponding to $\lambda_i$, and the resolvent $R_i(A)$ the unique real symmetric matrix inverting $\lambda_i-A$ on the range of $I-P_i(A)$, and vanishing on the range of $P_i(A)$. $R_i(A)$ can be written explicitly as
\begin{align*}
R_i(A)=\sum_{j: j\neq i} \frac{1}{\lambda_i-\lambda_j}u_ju_j^*.
\end{align*} 
Moreover $R_i(A)$ can be written as the following contour integral,
\begin{align}
\label{cont}R_i(A)=\frac{1}{2\pi i} \oint_{|z-\lambda_i|=\omega} \frac{1}{\lambda_i -z}\frac{1}{A-z}dz,
\end{align}
where we pick the contour $\omega$ to enclose only $\lambda_i$. From the above formula it is clear that $R_i(A)$ is a smooth function on a neighbourhood of $A$ if $\lambda_i$ is a single eigenvalue. We refer the reader to the book \cite[Chapter XII]{ReSi} for related properties.

If $\lambda_i$ is a single eigenvalue of $A$, we define the quantity 
\begin{align}\label{Q}
Q_i(A)=\frac{1}{N^2}\Tr(R_i(A)^2)=\frac{1}{N^2}\sum_{j:j\neq i}\frac{1}{|\lambda_j-\lambda_i|^2}.
\end{align}
This quantity plays an important role in \cite{tao2010random}, where it was observed that it captures quantitatively the derivatives of $\lambda_i(A)$. Here we write it in terms of the Green function, and prove it is stable under the DBM \eqref{DBM}. As a result, based on the idea of Green function comparison, it can be used to derive the weak level repulsion estimate Theorem \ref{tlevp}, once we know such an estimate for larger times, see Theorem \ref{Thlevpt}.

Since $Q_i(A)$ is not well-defined on the space of real symmetric matrices (it will blow up when $\lambda_i(A)$ is not a single eigenvalue), we have to compose it with a cutoff function $\chi_M$, where $M:=N^{2\tau}$ and the (small) constant $\tau>0$ will be chosen later. We choose the cutoff function $\chi_M(x)$ which satisfies the following two properties: (1) $\chi_M$ is smooth, and the first three derivatives are bounded by some constant $C$, i.e. $|\chi_M'(x)|, |\chi_M''(x)|, |\chi_M'''(x)|\leq C$. (2) On the interval $[0,M]$, $|\chi_M(x)-x|\leq 1$, and  for $x\geq M$, $\chi_M(x)=M$.

If $\lambda_i$ is a single eigenvalue of $A$, then in a neighborhood of $A$, $\chi_{M}(Q_i(A))$ is smooth; if $\lambda_i$ is not a single eigenvalue of $A$, then in a neighborhood of $A$, $\chi_{M}(Q_i(A))$ is constant, which is also smooth. Therefore $\chi_{M}(Q_i(A))$ is a well defined smooth function on the space of real symmetric matrices.

Our proof of Theorem \ref{tlevp} consists of three steps:
\begin{enumerate}
\item We use the level repulsion estimate for $H_t$ to conclude the estimate $\EE[\chi_M(Q_i(H_t))]\leq CN^{3\tau/2}$ for $t=N^{-1+\epsilon}$.
\item We compare $\EE[\chi_M(Q_i(H_t))]$ for $t=N^{-1+\epsilon}$ and $\EE[\chi_M(Q_i(H_0))]$ using the continuity Lemma \ref{compare}. Since $\eps>0$ can be chosen arbitrarily small we will obtain
\begin{align}
|\EE[\chi_M(Q_i(H_t))]-\EE[\chi_M(Q_i(H_0))]|\leq CN^{\tau}.
\end{align}
\item Theorem \ref{tlevp} then immediately follows from $\EE[\chi_M(Q_i(H_0))]\leq CN^{3\tau/2}$ and the Markov inequality.
\end{enumerate}
In the remainder of this section we denote the eigenvalues of $H_t$ by $\lambda_1(t)\leq \lambda_2(t)\leq \cdots \leq \lambda_N(t)$,  and so the eigenvalues of $H$ are $\lambda_1(0)\leq \lambda_2(0)\leq \cdots \leq \lambda_N(0)$.

The following level repulsion estimate is an immediate consequence of \cite[Theorem 3.6]{ly}.
\begin{theorem}\label{Thlevpt}
Let $H$ be a stable random matrix as defined in Section \ref{def23}. Given any (small) number $\omega >0$, and (large) number $D>0$, we have that
\begin{align}\label{levpt}
\PP\left(\left|\lambda_{i}(t)-\lambda_{i+1}(t)\right|\leq \theta N^{-1}\right)\leq N^{\omega}\theta^{2-\omega}+ N^{-D}
\end{align}
for any $i\in[[\kappa N, (1-\kappa)N]]$ and $t\geq N^{-1+\epsilon}$.
\end{theorem}
From this we derive the following estimate.
\begin{lemma}
Let $H$ be a stable random matrix.  Then,
\begin{align}
\EE [ \chi_M ( Q_i (H_t ) ) ] \leq C N^{3 \tau/2}
\end{align}
for $t = N^{-1+\eps}$ for any $\eps >0$ and all small $\tau >0$.
\end{lemma}
\begin{proof}
By our assumption with probability larger than $1-N^{-D}$, the matrix $H_t$ is $\delta$-general in the sense of Definition \ref{general}. Combining this with Theorem \ref{Thlevpt}, we have 
\begin{align*}
\PP\left(H_t \text{ is $\delta$-general and } \left|\lambda_{i}(t)-\lambda_{i\pm1}(t)\right|\geq \theta N^{-1}\right)\geq 1-N^{\omega}\theta^{2-\omega}-N^{-D}.
\end{align*} 
Define a dyadic decomposition:
\begin{align*}
&U_0=\{j\neq i:|\lambda_j(t)-\lambda_i(t)|\leq  N^{-1+\delta}\},\quad U_{\infty}=\{j: N^{\delta}< |\lambda_j(t)-\lambda_i(t)|\}\\
&U_n=\{j: 2^{n-1}N^{-1+\delta}< |\lambda_j(t)-\lambda_i(t)|\leq 2^n N^{-1+\delta}\}, \quad 1\leq n\leq \lceil\log_2N\rceil.
\end{align*}
For those $H_t$ which are $\delta$-general in the sense of Definition \ref{general}, $|U_n|\leq C2^n N^{\delta}$, for $0\leq n\leq \lceil\log_2 N\rceil$.  On the event that $H_t$ is $\delta$-general and $| \lambda_i(t) - \lambda_{i \pm 1}(t) | \geq \theta N^{-1}$, we derive the estimate,
\begin{align}
\label{ddsum}Q_i(H_t)
=&  \frac{1}{N^2}\sum_{n=0}^{\lceil \log_2 N\rceil}\sum_{j\in U_n}\frac{1}{|\lambda_i(t)-\lambda_j(t)|^2} + \frac{1}{N^{1+2\delta}}\leq 3CN^{\delta}\theta^{-2}.
\end{align}
Therefore we have
\begin{align*}
\PP\left(\frac{1}{N^2}\sum_{j: j\neq i} \frac{1}{|\lambda_i(t)-\lambda_j(t)|^2}\geq 3CN^{\delta}\theta^{-2}\right)\leq N^{\omega}\theta^{2-\omega}+N^{-D}
\end{align*}
where $\omega>0$ can be any small number. Therefore
\begin{align*}
\EE[\chi_M(Q_i(H_t))]\leq 1+3CN^{\delta}\theta^{-2}+M(N^{\omega}\theta^{2-\omega}+ N^{-D}).
\end{align*}
If we take $\theta=N^{-\tau/2}$, then we get $\EE[\chi_M(Q_i(H_t))]\leq CN^{3\tau/2}$, if we take $\delta<\tau/2$, $\omega< \tau/4$ and $D$ large. 
\end{proof}

This finishes the proof of the first step.  In order to apply Lemma \ref{compare} for the second step, we need to control the third derivative of $\chi_M(Q_i(\theta^{ab}H_s))$.

\begin{proposition}\label{bdrQ}
Let $A$ be an $N\times N$ deterministic real symmetric matrix. 
If $A$ is $\delta$-general in the sense of Definition \ref{general} and 
\begin{align}
\label{gapb}Q_i(A)=\frac{1}{N^2}\sum_{j:j\neq i}\frac{1}{|\lambda_{i}(A)-\lambda_{j}(A)|^2}\leq M=N^{2\tau},
\end{align} then 
\begin{align}\label{drQ}
\del_{ab}^{(k)}Q_i(A)\leq CN^{(2k+2)\delta+(k+2)\tau},\quad k=1,2,3.
\end{align}
for some constant $C$.
\end{proposition}
\begin{proof}
We denote $G=(A-z)^{-1}$ the resolvent of $A$ and by $\lambda_j$ and $u_j$ the eigenvalues and eigenvectors of $A$. Notice that \eqref{gapb} implies $|\lambda_{i}-\lambda_{i\pm1}|\geq N^{-1-\tau}$. The same dyadic argument leading to \eqref{ddsum} yields
\begin{align}
\label{gfirst}\sum_{j:j\neq i}\frac{1}{|\lambda_{i}-\lambda_{j}|}=\sum_{n\geq 0}\sum_{j\in U_n}\frac{1}{|\lambda_i-\lambda_j|}\leq CN^{1+\tau+\delta}.
\end{align} 
Also we have the trivial bound for higher moments
\begin{align}
\label{ghigh}\sum_{j:j\neq i}\frac{1}{|\lambda_{i}-\lambda_{j}|^k}\leq \left(\sum_{j:j\neq i}\frac{1}{|\lambda_{i}-\lambda_{j}|^2}\right)^{k/2}\leq N^{k(1+\tau)},\quad k\geq 2.
\end{align}
We denote by $V$ the matrix whose matrix elements are zero everywhere except at the $(a,b)$ and $(b,a)$ position, where it equals one. From the formula \eqref{cont}, 
\begin{align}
\label{der}\del_{ab}^{(k)}Q_i(A)=\frac{1}{N^2}\Tr \del_{ab}^{(k)}\oint \frac{G}{\lambda_i -z}dz\oint \frac{G}{\lambda_i -w}dw.
\end{align}
Notice that $\del_{ab}^{(k)}G= (-1)^k k! (GV)^kG$. By the Leibniz rule, \eqref{der} can be written as a sum of terms in the following form
\begin{align}
\label{term}\frac{1}{N^2}\Tr \oint (GV)^{k_1}G\del_{ab}^{(k_2)}\frac{1}{\lambda_i-z}dz\oint (GV)^{k_3}G\del_{ab}^{(k_4)}\frac{1}{\lambda_i -w}dw,
\end{align}
where $k_1+k_2+k_3+k_4=k$. We will only prove \eqref{drQ} for $k=3$; that is, we will prove 
\begin{align}
|\del_{ab}^{(3)}Q_i(A)|\leq CN^{5\tau+8\delta} \label{eqn:threeder}
\end{align}
The computations for $k=1,2$ are much easier. To evaluate \eqref{term}, we need to compute the first three derivatives of $\lambda_i(A)$ with respect to the $(a,b)$-th entry of $A$. We use the following formula to compute the derivatives of $\lambda_i$,
\begin{align*}
\lambda_i=-\frac{1}{2\pi i} \Tr \oint \frac{z}{A-z} dz
\end{align*}
where the contour encloses only $\lambda_i$.  The $k$-th derivative with respect to $(a,b)$-th entry is
\begin{align*}
\del_{ab}^{(k)}\lambda_i=-\frac{1}{2\pi i} \Tr \oint z\del_{ab}^{(k)}G dz=\frac{(-1)^{k+1}k!}{2\pi i} \Tr\oint z(GV)^kGdz.
\end{align*}

For $k=1,2,3$ respectively we have
\begin{align}
\label{fd}\del_{ab}\lambda_i&=\frac{1}{2\pi i} \sum_{j=1}^{N}\oint \frac{zu_j^* V u_j}{(\lambda_j-z)^2}dz=u_i^* V u_i,\\
\label{sd}\del_{ab}^{(2)}\lambda_i&=\frac{-2}{2\pi i} \sum_{j_1,j_2} \oint \frac{z(u_{j_1}Vu_{j_2}^*)^2}{(\lambda_{j_1}-z)^2 (\lambda_{j_2}-z)}=2\sum_{j:j\neq i}\frac{(u_i^* Vu_j)^2}{\lambda_i-\lambda_j},\\
\notag\del_{ab}^{(3)}\lambda_i&=\frac{6}{2\pi i} \sum_{j_1,j_2,j_3} \oint \frac{z(u_{j_1}Vu_{j_2}^*)(u_{j_2}Vu_{j_3}^*)(u_{j_3}Vu_{j_1}^*)}{(\lambda_{j_1}-z)^2 (\lambda_{j_2}-z)(\lambda_{j_3}-z)}\\
\label{td}&=6\sum_{j_1, j_2:\atop j_1, j_2\neq i}\frac{(u_i^* Vu_{j_1})(u_{j_1}^*V u_{j_2})(u_{j_2}^*Vu_{i})}{(\lambda_i-\lambda_{j_1})(\lambda_i-\lambda_{j_2})}
-6(u_i^*V u_i)\sum_{j:j\neq i} \frac{(u_i^*Vu_j)(u_j^*Vu_i)}{(\lambda_i-\lambda_j)^2}
\end{align}

Some straightforward but tedious integration similar to \eqref{fd}, \eqref{sd},\eqref{td} reveals that $\del_{ab}^{(3)}Q_i(A)$ is a sum of the following terms (for simplicity of notation we write $V_{jk}:=u_j^* V u_k$):
\begin{align*}
&\sum_{j: j\neq i}\frac{\Omega}{(\lambda_i-\lambda_j)^5},\quad \Omega=V_{ii}^3, V_{ii}^2V_{jj}, V_{ii}V_{jj}^2, V_{jj}^3.\\
&\sum_{j_1,j_2:\atop j_1,j_2\neq i}\frac{\Omega}{(\lambda_i-\lambda_{j_1})^4(\lambda_{i}-\lambda_{j_2})},\quad \Omega= V_{ii}V_{j_1j_2}^2,V_{j_1j_1}V_{ij_2}^2,V_{j_1j_1}V_{j_1j_2}^2, V_{ii}V_{ij_2}^2.\\
&\sum_{j_1,j_2:\atop j_1,j_2\neq i}\frac{\Omega}{(\lambda_i-\lambda_{j_1})^3(\lambda_{i}-\lambda_{j_2})^2},\quad
\Omega= V_{j_2j_2}V_{j_1j_2}^2, V_{ii}V_{ij_1}^2, V_{ii}V_{j_1j_2}^2, V_{j_1j_1}V_{j_1j_2}^2.\\
&\sum_{j_1,j_2,j_3:\atop j_1,j_2,j_3\neq i}\frac{V_{j_1j_2}V_{j_2j_3} V_{j_3j_1}}{(\lambda_i-\lambda_{j_1})^3(\lambda_i-\lambda_{j_2})(\lambda_i-\lambda_{j_3})},
\sum_{j_1,j_2,j_3:\atop j_1,j_2,j_3\neq i}\frac{V_{j_1j_2}V_{j_2j_3} V_{j_3j_1}}{(\lambda_i-\lambda_{j_1})^2(\lambda_i-\lambda_{j_2})^2(\lambda_i-\lambda_{j_3})}.
\end{align*}
Indeed, if one takes a close look at the expression \eqref{term}, the only singularity enclosed by our contour is $\lambda_i$. Therefore by Cauchy's formula, the integral is sum of terms with denominators: $\prod_j (\lambda_i-\lambda_j)$, as appearing in the above expressions. 

Since the eigenvectors of $A$ are completely delocalized we see that
\begin{align}
|V_{jk} | \leq C N^{-1+\delta}.
\end{align}
This, together with the bounds \eqref{gfirst} and \eqref{ghigh} yields \eqref{eqn:threeder}
\end{proof}

\subsection{Proof of Theorem \ref{tlevp}}
By explicit computation, we have
\begin{align}
\label{third}\del_{ab}^{(3)}\left(\chi_M(Q_i(\theta^{ab}H_s))\right)&=\chi_M''' (\del_{ab}Q_i)^3+3\chi_M'' \del_{ab}^{(2)}Q_i\del_{ab}Q_i+\chi_M' \del_{ab}^{(3)}Q_i.
\end{align}
Notice that the right hand side of \eqref{third} vanishes unless $Q_i(\theta^{ab}H_s)\leq M$. Since $H$ is stable, $\theta^{ab}H_s$ satisfies all the assumptions in Lemma \ref{drQ}, with probability larger than $1-N^{-D}$ for any large number $D$. Therefore
\begin{align*}
\PP\left(\left| \del_{ab}^{(3)}\left(\chi_M(Q_i(\theta^{ab}H_s))\right)\right|\leq CN^{12\delta+9\tau} \right)\geq 1-N^{-D}.
\end{align*}
On the complement of the above event we will use the following deterministic bound 
\begin{align*}
\left| \del_{ab}^{(3)}\left(\chi_M(Q_i(\theta^{ab}H_s))\right)\right|\leq CN^{3+9\tau}.
\end{align*}
We want to apply Lemma \ref{compare} with $F = \chi_M \circ Q_i$. For this choice of $F$ we can take $B$ to satisfy
\begin{align*}
B\leq \sup_{0\leq s\leq t\atop 1\leq i,j\leq N} \EE\left[(N^2 |h_{ij}(s)-\f|^3+N|h_{ij}(s)-\f|) CN^{12\delta+9\tau}\right]+CN^{3+9\tau}N^{-D}\leq CN^{1+12\delta+9\tau-\alpha},
\end{align*}
where the last factor $N^{-\alpha}$ is from the third moment of $h_{ij}(s)$. We conclude that
\begin{align}
\label{bd} |\EE[\chi_M(Q_i(H_t))]-\EE[\chi_M(Q_i(H_0))]|\leq Ct N^{1+12\delta+9\tau-\alpha}= C N^{\epsilon+12\delta+9\tau-\alpha}.
\end{align}
Since the two numbers $\epsilon,\delta>0$ can be arbitrarily small, therefore we can choose $\tau=(\alpha-\epsilon-12\delta)/8$ and \eqref{bd} simplifies to
\begin{align}
|\EE[\chi_M(Q_i(H_t))]-\EE[\chi_M(Q_i(H_0))]|\leq CN^{\tau}.
\end{align}
Hence,
\begin{align}
\EE[\chi_M(Q_i(H_0))] \leq C N^{ 3 \tau/2}
\end{align}
and the proof is easily concluded by the Markov inequality.
\qed

\section{Bulk universality of stable random matrices }\label{bulkH}
In this section we prove bulk universality for stable random matrices $H$, i.e., Theorem \ref{TbulkH} by comparing the local statistics between $H$ and $H_t$.   This will yield the theorem as Theorem \ref{thm:htbulk} shows that the latter ensemble exhibits bulk universality. In the following of this section we denote the eigenvalues of $H_t$ by $\lambda_1(t)\leq \lambda_2(t)\leq \cdots \leq \lambda_N(t)$. And so the eigenvalues of $H$ are $\lambda_1(0)\leq \lambda_2(0)\leq \cdots \leq \lambda_N(0)$.

We will obtain the gap universality of $H$ from the more general comparison result below.
\begin{lemma} \label{lem:comparegap} Let $H$ be a stable random matrix and let $H_t$ be defined as in \eqref{DBM}.  Take $t = N^{-1+\eps}$.  Then for $\eps >0$ small enough we have
\begin{align}\label{compg}
\lim_{N\rightarrow \infty}\EE[O(N\lambda_{i_1}(0), \cdots, N\lambda_{i_n}(0))]-\EE[O(N\lambda_{i_1}(t), \cdots, N\lambda_{i_n}(t))]=0,
\end{align} 
for any $i_1,i_2,\cdots, i_n\in [[\kappa N, (1-\kappa)N]]$ and bounded test function $O\in C^{\infty} (\RR^n)$ whose first three derivatives are bounded.
\end{lemma}
\proof 
For simplicity of notation, we only state the proof for $n=1$ case, i.e. for any $i\in [[\kappa N, (1-\kappa)N]]$
\begin{align}\label{oneg}
\lim_{N\rightarrow \infty}\EE[O(N\lambda_{i}(0))]-\EE[O(N\lambda_{i}(t))]=0.
\end{align} 
  
Take a cutoff function $\rho_M$ such that $\rho_M(x)=1$ for $x\leq M$ and $\rho_M(x)=0$ for $x\geq 2M$, where $M=N^{2\tau}$ and $\tau>0$ is a small constant. By the level repulsion of $H$ and $H_t$ from the previous section, we know that
\begin{align*}
\PP(Q_i(H_s)\geq N^{2\tau})\leq N^{-\tau/2}, \quad s=0, t.
\end{align*} 
Since $O\in C^{\infty}(\RR)$ is bounded,  we have that
\begin{align*}
&\left|\EE[O(N\lambda_{i}(0))]-\EE[O(N\lambda_{i}(t))]\right|\\
\leq &\left|\EE[O(N\lambda_{i}(0))\rho_{M}(Q_i(H_0))]-\EE[O(N\lambda_{i}(t))\rho_{M}(Q_i(H_t))]\right|\\
+&\|O\|_{\infty}\left(\PP(Q_i(H_0)\geq N^{2\tau})+\PP(Q_i(H_t)\geq N^{2\tau})\right)\\
\leq &\left|\EE[O(N\lambda_{i}(0))\rho_{M}(Q_i(H_0))]-\EE[O(N\lambda_{i}(t))\rho_{M}(Q_i(H_t))]\right|+\frac{2\|O\|_{\infty}}{N^{\tau/2}}.
\end{align*}
Notice that $O(N\lambda_{i}(A))\rho_{M}(Q_i(A))$ is a well defined smooth function on the space of symmetric functions. Moreover, if the matrix $A$ is $\delta$-general in the sense of Definition \ref{general} 
the same argument as in Proposition \ref{bdrQ} implies
\begin{align*}
\left|\del_{ab}^{(3)}O(N\lambda_{i}(A))\rho_{M}(Q_i(A))\right|\leq CN^{c(\delta+\tau)},
\end{align*}
where $c$ and $C$ are constants. Therefore by Lemma \ref{compare}, we have
\begin{align*}
\left| \EE[O(N\lambda_{i}(0))\rho_{M}(Q_i(H_0))]-\EE[O(N\lambda_{i}(t))\rho_{M}(Q_i(H_t))] \right| \leq CN^{c(\delta+\tau)+\epsilon-\alpha}\rightarrow 0,
\end{align*}
if we take $c(\delta+\tau)+\epsilon<\alpha/2$.
\qed

For the universality of the correlation functions, due to the lack of optimal rigidity, one can not directly deduce universality from \eqref{compg}. We need to prove the following Green function comparison lemma:

\begin{lemma}\label{gcp}
Let $H$ be a stable random matrix as defined in Section \ref{def23}, and let $H_t$ be defined as in \eqref{DBM}. Let $\delta>0$ be arbitrary and choose an $\eta$ with $N^{-1-\delta}\leq \eta \leq N^{-1}$. For any sequence of positive integers $k_1,k_2,\dots,k_n$, any set of complex parameters
$z_j^m = E_j^m\pm \eta$, where $1\leq j\leq k_m$, $1\leq m\leq n$, $|E_j^m| \leq  2-\kappa$,
and the $\pm$ signs are arbitrary, we have the following. Let $G_{t}(z)=(H_t-z)^{-1}$ be the resolvent and let $F(x_1,x_2,\cdots, x_n)$ be a test function such that for any multi-index $\alpha=(\alpha_1,\cdots, \alpha_n)$ with $1\leq |\alpha|\leq 4$ and for any $\omega>0$ sufficiently small, we have
\begin{align}\label{reg1}
\max\left\{|\partial^\alpha F(x_1,x_2,\cdots, x_n)|: \max_{j} |x_{j}|\leq N^{\omega}\right\}\leq N^{C_0 \omega}
\end{align}
and
\begin{align}\label{reg2}
\max\left\{|\partial^\alpha F(x_1,x_2,\cdots, x_n)|: \max_{j} |x_{j}|\leq N^{2}\right\}\leq N^{C_0}
\end{align}
for some constant $C_0$. Then the following holds: 
\begin{align}\label{gcomp}
\left|\EE[F\left(N^{-k_1} \Tr \prod_{j=1}^{k_1}G_t(z_j^1), \cdots, N^{-k_n} \Tr \prod_{j=1}^{k_n}G_t(z_j^n)\right)- \EE[F(G_t\rightarrow G_0)]]\right|\leq C t N^{1 +c \delta-\alpha},
\end{align}
where $c$ and $C$ are constants depending on $\kappa$, $n$, $k_1,k_2,\cdots, k_n$ and $C_0$.  The second term above is the same as the first term, but with $G_0$ replacing $G_t$ everywhere.
\end{lemma}
\begin{proof}
For simplicity of notation, we state the proof only for $n=1$ and $k_1=1$ case, i.e. 
\begin{align*}
\left|\EE[F\left( N^{-1}\Tr G_t(z)\right)]- \EE[F\left( N^{-1}\Tr G_0(z)\right)]\right|\leq C t N^{1 +c \delta-\alpha}.
\end{align*}

We will prove this lemma using Lemma \ref{compare}.  We must compute derivatives of the trace of the Green's function of the deformed matrix $\theta^{ab} H_s$.  We denote the resolvent of $\theta^{ab}H_s$ by $G(z)=(\theta^{ab}H_s-z)^{-1}$. For the derivatives of $G(z)$, we have  $|\Tr \del_{ab}^{(k)} G|=|(-1)^k k!\Tr (GV)^kG|$, where $V$ is the matrix whose matrix elements are zero everywhere except at the $(a,b)$ and $(b,a)$ position, where it equals one. Since $V$ has at most two nonzero elements,  the trace  $(GV)^kG$ contains at most $2^k N$ terms.  Furthermore, each term is a product of $k+1$ entries of $G$, e.g. $\Tr GVG=\sum_{k}(G_{kb}G_{ab}+G_{ka}G_{bk})$.  
We first derive a bound on the resolvent entries $G_{jk}(E+i\eta)$ down to $\eta\geq N^{-1-\delta}$ when $\theta^{ab} H_s$ is $\delta$-general. Using the delocalization of the eigenvectors we have,
\begin{align*}
|G_{jk}(E+i\eta)|\leq \sum_{i=1}^{N}\frac{|u_i(j)u_i(k)|}{|\lambda_i-z|}\leq CN^{-1+\delta}\sum_{i=1}^{N}\frac{1}{|\lambda_i-z|}.
\end{align*}
Define a dyadic decomposition:
\begin{align*}
&U_0=\{j:|\lambda_j-E|\leq  N^{-1+\delta}\},\quad U_{\infty}=\{j: N^{\delta}< |\lambda_j-E|\}\\
&U_n=\{j: 2^{n-1}N^{-1+\delta}< |\lambda_j-E|\leq 2^n N^{-1+\delta}\}, \quad 1\leq n\leq \lceil\log_2N\rceil.
\end{align*}
For $\delta$-general $\theta^{ab}H_s$ we have $|U_n|\leq C2^n N^{\delta}$, for $0\leq n\leq \lceil\log_2 N\rceil$. We can divide the summation over $i$ into $\cup_n U_n$,
\begin{align*}
|G_{jk}(E+i\eta)|
\leq&  CN^{-1+\delta}\sum_{n\geq 0}\sum_{i\in U_n}\frac{1}{|\lambda_i-E-i\eta|}\leq CN^{3\delta},\quad N^{-1-\delta}\leq \eta.
\end{align*}
Therefore we have 
\begin{align*}
\PP(N^{-1}|\Tr \del_{ab}^{(k)} G|\leq CN^{3(k+1)\delta})\geq 1-N^{-D} , \quad N^{-1-\delta}\leq \eta
\end{align*}
When $\theta^{ab}H_s$ is not $\delta$-general we still have the deterministic upper bound
\begin{align*}
N^{-1}|\Tr \del_{ab}^{(k)} G|\leq CN^{3(1+\delta)}.
\end{align*}
Therefore we can take $F(A)= N^{-1}\Tr (A-E-i\eta)^{-1}$ for $N^{-1-\delta}\leq \eta$ in Lemma \ref{compare}.  For $B$ we have the upper bound
\begin{align*}
B\leq &\sup_{0\leq s\leq t,1\leq i,j\leq N} \EE\left[(N^2 |h_{ij}(s)-\f|^3+N|h_{ij}(s)-\f|) CN^{(3C_0+18)\delta}\right]+CN^{3(1+\delta)}N^{-D}\\
\leq &CN^{1+(3C_0+18)\delta-\alpha},
\end{align*}
which yields \eqref{gcomp}.
\end{proof}

Once we have the above lemma, the following theorem from \cite[Theorem 2.1]{ESYY} transforms the information of the Green function to the correlation functions of $H$ and will complete the proof of Theorem \ref{TbulkH}.
\begin{theorem}\label{thm:comparecor}
Let $G(z)$ and $G_t(z)$ denote the Green function of the two matrices $H$ and $H_t$, respectively. Suppose that \eqref{gcomp} holds for the two random matrices $H$ and $H_t$ for any $t=N^{-1+\epsilon}$, where $\epsilon>0$ can be arbitrarily small. Let $\rho_{H}^{(n)}$ and $\rho_{H_t}^{(n)}$ be the $n$-point correlation functions of the eigenvalues w.r.t. the probability laws of the matrices $H$ and $H_t$. Then for any $\kappa>0$, $E\in (-2+\kappa, 2-\kappa)$,  and test function $O\in C_{c}^{\infty}(\RR^n)$ we have 
\begin{align*}
\lim_{N\rightarrow \infty}\int_{\RR^n} O(\alpha_1,\cdots, \alpha_n)\left\{ \rho^{(n)}_{H}\left( E+\frac{\alpha_1}{N},\cdots E+\frac{\alpha_n}{N}\right)-\rho^{(n)}_{H_t}\left( E+\frac{\alpha_1}{N},\cdots E'+\frac{\alpha_n}{N}\right)\right\}=0.
\end{align*}
provided that $t$ is chosen so that $C t N^{1 +c \delta-\alpha}\leq N^{-\alpha/2}$.
\end{theorem}

\subsection{Proof of Theorem \ref{TbulkH}}
The universality of the gap statistics follows from Theorem \ref{thm:htbulk} and Lemma \ref{lem:comparegap}. The universality of the correlation functions follows from Theorem \ref{thm:htbulk}, Lemma \ref{gcp} and Theorem \ref{thm:comparecor}.
\qed

\section{Universality of sparse random matrices}\label{example}
In this section we prove Theorem \ref{bsparse}, the bulk universality of sparse matrices, by checking that sparse  matrices satisfy the hypotheses of Definition \ref{stabledef} of stable random matrices. In the following we collect some facts about sparse matrices proved in \cite{EKYYa}.

\begin{theorem}\label{locallaw}
Let $H$ be a sparse Wigner matrix as in Definition \ref{def:snrm}. We denote the eigenvalues of $H$ as $\lambda_1\leq \lambda_2\leq \cdots\leq \lambda_N$, the corresponding eigenvectors $u_1,u_2,\cdots, u_N$, and the resolvent $G(E+i\eta)=(H-E-i\eta)^{-1}$. Then for all (small) $\omega, \delta>0$, (large) $D>0$ and large enough $N\geq N(\omega,\delta,D)$ the following holds with probability larger than $1-N^{-D}$.
\begin{enumerate}[label={\upshape(\roman*)}]
\item All eigenvalues of $H$ are in the interval $[-3,3]$, i.e. $-3\leq \lambda_1\leq \lambda_2\cdots\leq \lambda_{N-1}\leq 3$, except for the largest eigenvalue $\lambda_N$. 
\item We have the weak local semi-circle law for individual resolvent entries: there exists a constant $C$ such that $|G_{jk}(E+i\eta)|\leq C$, for $\eta \geq N^{-1+\delta}$\label{Gbound}
\item We have the local semi-circle law for the Stieltjes transform of eigenvalues of $H$:
\begin{align}\label{locLaw}
\left|m_N(E+i\eta)-m_{sc}(E+i\eta)\right|\leq N^{\omega} \left( \frac{1}{q}+\frac{1}{N\eta}\right),\quad |E|\leq 5, N^{-1+\delta}\leq \eta\leq 10.
\end{align}
\end{enumerate}
\end{theorem}

\begin{remark}
In fact Theorem \ref{locallaw} is still true for generalized sparse matrices, in which we allow the variance of each entry $\EE[(h_{ij}-f)^2]=s_{ij}$ to be different, given that: (1) they are of the same order, i.e. there exists some constants $c_1$ and $c_2$, such that $c_1N^{-1}\leq s_{ij}\leq c_2 N^{-1}$, (2) $\sum_{j=1}^{N}s_{ij}=1$.
\end{remark}

\begin{theorem} Sparse random matrices in the sense of Definition \ref{def:snrm} are stable.  Therefore, Theorem \ref{bsparse} holds.

\end{theorem}

\begin{proof}
Conditions \ref{ind} and \ref{moment} of Definition \ref{stabledef}, i.e., the independent entries and moment conditions, follow from the definition of a sparse matrix. The fact that $\theta^{ab}H_s$ is $\delta$-general, i.e., condition \ref{stable}, will follow from the bounds \ref{Gbound} in Theorem \ref{locallaw}. If the resolvent elements of $\theta^{ab}H_s$ are bounded down to the scale $\eta\geq N^{-1+\delta}$ then by taking $E=\lambda_i$ and $\eta=N^{-1+\delta}$ in the following identity
\begin{align}
C\geq \label{resod}\text{Im } G_{jj}(E+i\eta)=\sum_{i=1}^{N} \frac{\eta |u_i(j)|^2}{(\lambda_{i}-E)^2+\eta^2}\geq |u_i(j)|^2\eta^{-1},
\end{align}
we see that $|u_i(j)|^2\leq CN^{-1+\delta}$ and so the eigenvectors of $\theta^{ab}H_s$ are completely delocalized. For any interval $I=[E-\eta, E+\eta]$, such that $N^{-1+\delta}\leq \eta$, we have
\begin{align*}
C\geq \text{Im } m(E+i\eta)=\frac{1}{N} \sum_{i=1}^{N} \frac{\eta}{(\lambda_i-E)^2+\eta^2}
\geq \frac{1}{N} \sum_{\lambda_i\in I} \frac{\eta}{(\lambda_i-E)^2+\eta^2}
\geq \frac{1}{|I| N}\#\{i: \lambda_i\in I\}.
\end{align*}
Therefore, we get that $\#\{i: \lambda_i\in I\}\leq C|I|N$ by rearranging the above expression.

It therefore suffices to prove that the resolvent entries of the deformed matrix $\theta^{ab} H_s$ are bounded down to the scale $\eta = N^{-1+\delta}$.  For each $s$, $H_s$ is a sparse matrix and so by Theorem \ref{locallaw} we know that 
its resolvent entries are bounded with probability greater than $1-N^{-D}$ for any large $D$.

The deformed matrix $\theta^{ab}H_s$ is a rank two perturbation of $H_s$, i.e. $\theta^{ab}H_s=H_s-V$, where $V$ is the matrix whose matrix elements are zero everywhere except at the $(a,b)$ and $(b,a)$ position, where it equals $(1-\theta^{ab}_{ab})h_{ab}(s)$. By our assumption on the moments of sparse Wigner matrix \eqref{momentbound}, we have that
\begin{align}
\PP \left( |(1-\theta^{ab}_{ab})h_{ab}(s)|\leq N^{-\alpha/2} \right) \geq 1 - N^{-D}
\end{align}
for any large $D$. We denote the resolvent of $H_s$ as $G=(H_s-z)^{-1}$. The resolvent elements of $\theta^{ab}H_s$ are given by the formula
\begin{align*}
|(\theta^{ab}H_s-z)_{jk}^{-1}|=&\left|G_{jk}+(GV(\theta^{ab}H_s-z)^{-1})_{jk}\right|\\
\leq& C+2CN^{-\alpha/2} \max_{1\leq n,l\leq N}|(\theta^{ab}H_s-z)^{-1}_{nl}|\\
\leq& C+ \frac{1}{2}\max_{1\leq n,l\leq N}|(\theta^{ab}H_s-z)^{-1}_{nl}|
\end{align*}
with probability greater than $1 - N^{-D}$, for $N$ sufficiently large. We get $\max_{j,k}|(\theta^{ab}H_s-z)_{jk}^{-1}|\leq 2C$, by taking maximum over $j$ and $k$ in above estimate, and rearranging it. This finishes the proof that with probability larger than $1-N^{-D}$ for any (large) number $D$, $\theta^{ab}H_s$ is general. 

For the local semi-circle law of $H_t^{(1)}$, i.e., condition \ref{loclaw}, the variance of the diagonal terms and off-diagonal terms of $H_t^{(1)}$ are $e^{-t}N^{-1}$ and  $(1+e^{-t})/2N$ respectively. Therefore after normalizing by $(1-(1-e^{-t})\frac{N+1}{2N})^{-1}= 1+O(t)$, it is a generalized sparse matrix. By \eqref{msc}, the normalization factor gives us an error of order at most $O(\sqrt{t})$: with probability larger than $1-N^{-D}$,
\begin{align*}
\left|m_N^{(1)}(z)-m_{sc}(z)\right|\leq N^{\omega} \left( \frac{1}{q}+\frac{1}{N\eta}+\sqrt{t}\right),\quad |E|\leq 4, N^{-1+\delta}\leq \eta \leq 10.
\end{align*}
where $m_N^{(1)}(z)$ is the Stieltjes transform of the empirical eigenvalue distribution of $H_t^{(1)}$. Therefore we can take $c(\delta)=\frac{1}{2}\min\{\alpha, \delta, \frac{1-\epsilon}{2}\}$ in \eqref{loclawHt1}.

We have checked that sparse random matrices are stable, and so Theorem \ref{bsparse} now follows from Theorem  \ref{TbulkH}.
\end{proof}


\bibliography{References}{}
\bibliographystyle{abbrv}

\end{document}